\documentclass[12pt, reqno]{amsart}
\usepackage{graphicx}
\usepackage{hyperref, cite}
\usepackage[caption = false]{subfig}
\usepackage{amsthm}
\usepackage{mathrsfs}
\usepackage{mathtools}
\usepackage{enumerate}
\usepackage{amssymb}
\usepackage{float}
\usepackage{wrapfig}
\usepackage[none]{hyphenat}

\usepackage[twoside,paperwidth=210mm, paperheight=297mm, top=35mm, bottom=20mm, left=20mm, right=20mm, marginparsep=3mm, marginparwidth=40mm]{geometry}

\numberwithin{equation}{section}
\DeclareMathOperator{\RE}{Re}
\DeclareMathOperator{\IM}{Im}
\theoremstyle{plain}

\newtheorem{theorem}{Theorem}[section]
\newtheorem{corollary}{Corollary}[section]

\newtheorem{lemma}{Lemma}[section]

\theoremstyle{definition}

\theoremstyle{remark}
\newtheorem{remark}{Remark}[section]
\makeatother

\title{Starlike Functions associated with a Petal Shaped Domain}

\begin{document}

\begin{center}
\maketitle
\thispagestyle{empty}
\pagestyle{empty}

\textbf{S. Sivaprasad Kumar\textsuperscript{1}} \and \textbf{Kush Arora\textsuperscript{2}}\\

\qquad \\
\textsuperscript{\textbf{1,2}}Department of Applied Mathematics\\
Delhi Technological University\\
Delhi-110042, India\\
\verb"spkumar@dce.ac.in"\\
\verb"kush.arora1214@gmail.com"\\
\end{center}

\begin{abstract}
This paper deals with some radius results and inclusion relations that 
are established for functions in a newly defined subclass of starlike functions associated with a petal shaped domain.
\end{abstract}


\section{Introduction}

Let the open unit disk $\{z\in\mathbb{C}:|z|<1\}$ be represented by $\mathbb{D}$ and denote the class of all analytic functions in $\mathbb{D}$ by $\mathcal{H}$. Consider $\mathcal{A}_n$ as the class of analytic functions $f$ in $\mathbb{D}$ represented by
\begin{equation}\label{A_n}
  f(z)=z+a_{n+1}z^{n+1}+a_{n+2}z^{n+2}+...
\end{equation}
In particular, denote $\mathcal{A}_1:=\mathcal{A}$ and let $\mathcal{S}$ be the subclass of $\mathcal{A}$ such that it involves all univalent functions $f(z)$ in $\mathbb{D}$. Let $g, h$ be two  analytic functions and $\omega$ be a Schwarz function satisfying $\omega(0)=0$ and $|\omega(z)|\leq|z|$ such that $g(z)=h(\omega(z))$ then $g$ is said to be subordinate to $h$, or $g \prec h$. If $h$ is univalent, then $g\prec h$ iff $g(0)=h(0)$ and $g(\mathbb{D})\subset h(\mathbb{D})$. Ma and Minda \cite{minda94} introduced the univalent function $\psi$ satisfying $\RE{\psi(\mathbb{D})}>0$, $\psi(\mathbb{D})$ starlike with respect to $\psi(0)=1$ and $\psi'(0)>0$ and the domain $\psi(\mathbb{D})$ being symmetric about the real axis. Further, they gave the definitions for the general subclasses of starlike and convex functions, respectively as follows: 
\begin{equation}
\mathcal{S}^*(\psi):=\left\{f\in\mathcal{S}: zf'(z)/f(z)\prec\psi(z)\right\}
\end{equation}
and
\begin{equation}
\mathcal{K}(\psi):= \left\{f\in\mathcal{S}: 1 + zf''(z)/f'(z) \prec \psi(z)\right\}.
\end{equation}
For different choices of $\psi$, many subclasses of $\mathcal{S}^*$ and $\mathcal{K}$ can be obtained. For example, the notable classes of Janowski starlike and convex functions \cite{janow} are represented by $\mathcal{S}^*[C,D] := \mathcal{S}^*((1+Cz)/(1+Dz))$ and $\mathcal{K}[C,D]:= \mathcal{K}((1+Cz)/(1+Dz))$ for $-1 \leq D < C \leq 1$, respectively. Further, $\mathcal{S}^*_{\alpha}:= \mathcal{S}^*[1-2\alpha,-1]$ and $\mathcal{K}_{\alpha}:= \mathcal{K}^*[1-2\alpha,-1]$ represents the classes of starlike and convex functions of order $\alpha \in [0,1)$, respectively. Note that $\mathcal{S}^*:= \mathcal{S}^*_0$ and $\mathcal{K}:=\mathcal{K}_0$ represent the well-known classes of starlike and convex functions, respectively. We denote $\mathcal{SS}^*(\gamma):= \mathcal{S}^*(((1+z)/(1-z))^\gamma)$ and $\mathcal{SK}(\gamma):= \mathcal{K}(((1+z)/(1-z))^\gamma)$ representing the class of strongly starlike and strongly convex functions of order $\gamma \in (0,1]$ respectively.

Recall that for two subfamilies $G_1$ and $G_2$ of $\mathcal{A}$, we say that $r_0$ is the $G_1$ -- radius for the class $G_2$, if $r_0$, $(0 < r \leq r_0)$, is the greatest number which satisfies $r^{-1}g(rz)\in G_1$ where $g \in G_2$. Moreover, starlike class $\mathcal{S}^*(\psi)$ for different $\psi(z)$ were considered by many authors, whose works examined the geometrical properties, radius results and coefficient estimates of the functions of their respective classes. Sok\'{o}\l \; and Stankiewicz \cite{sokol96,sokol09} considered the class $\mathcal{S}^*_{L}:= \mathcal{S}^*(\sqrt{1+z})$ and Mendiratta $et\; al.$ \cite{mendi} worked on the class $\mathcal{S}^*_{RL}:= \mathcal{S}^*(\sqrt{2} - (\sqrt{2}-1) ((1-z)/((1+2(\sqrt{2}-1)z)))^{1/2})$. Sharma $et\; al.$ \cite{naveen14} studied the class $\mathcal{S}^*_{C}:= \mathcal{S}^*(1+4z/3+2z^2/3)$ while the class $\mathcal{S}^*_s:=\mathcal{S}^*(1+\sin{z})$ was examined by Cho $et\; al.$ \cite{sinefun}. The classes $\mathcal{S}^*_{e}:=\mathcal{S}^*(e^z)$ and $\Delta^*:= \mathcal{S}^*(z+\sqrt{1+z^2})$ were considered by Mendiratta $et\; al.$ \cite{mendi2exp} and Raina $et\; al.$ \cite{raina}, respectively. Kargar $et\; al.$ \cite{kargarbooth} introduced and studied the class $\mathcal{BS}^*(\alpha) := \mathcal{S}^*(1+z/(1-\alpha z^2)), \; \alpha \in [0,1]$, associated with the Booth lemniscate which was also investigated by Cho $et\;al.$ \cite{chobooth}. Some more recent work on radius problems can be found in \cite{aghalary, chobell, bano, priyanka, swaminathan}. 

Motivated by the classes defined in \cite{sokol96,kargarbooth,mendi,naveen14,mendi2exp,sinefun,raina}, we consider the petal shaped region ${\Omega}_{\rho}:=\{ w\in \mathbb{C}: |\sinh(w - 1)| < 1 \}$, which is characterised functionally as $\rho(z)=1+\sinh^{-1}(z)$ to define our class. Clearly, $\rho(z)$ is a Ma-Minda function. See Figure {\ref{fig:incl_rel}} for its boundary curve $\gamma_0$ which is petal shaped. Note that $\sinh^{-1}(z)$ is a multivalued function and has the branch cuts along the line segments $(-i\infty,-i) \cup (i,i\infty)$, on the imaginary axis
and hence it is analytic in $\mathbb{D}$. Now we introduce a new class of starlike functions
\begin{equation}
\label{func_def}
\mathcal{S}^*_{\rho}:=\left\{f\in\mathcal{A}: \frac{zf'(z)}{f(z)} \prec 1+\sinh^{-1}(z)\right\} \quad (z\in\mathbb{D}),
\end{equation}
which is associated with the petal-shaped domain $\rho(\mathbb{D})$. From the above definition, we deduce that $f\in \mathcal{S}^*_{\rho} $ iff there exists an analytic function $q(z)\prec \rho(z)$ such that
\begin{equation}\label{gen_int_rep}
f(z) = z \exp\left(\int^z_0 \frac{q(t)-1}{t}dt\right).
\end{equation}
Table \ref{func_exa_table} presents some functions in the class $\mathcal{S}^*_{\rho}$ where $q_j \prec \rho$.

\begin{table}[ht]
\begin{center}
\begin{tabular}{ccc}
\hline
$j$ & $q_j(z)$ & $f_j(z)$\\
\hline
$1$ & $1 + z/5$ & $z \exp(z/5)$\\
$2$ & $(5+2z)/(5+z)$ & $z + z^2/5$\\
$3$ & $(7+4z)/(7+z)$ & $z(1+z/7)^3$\\
\hline
\end{tabular}
\end{center}
\caption{Some functions in the class $\mathcal{S}^*_{\rho}$}
\label{func_exa_table}
\end{table}

Since $\rho$ is univalent in $\mathbb{D},\, q_j(\mathbb{D}) \subset \rho(\mathbb{D})$ and $q_j(0)=\rho(0) \; (j=1, 2, 3)$, it follows that each $q_j \prec \rho$. Thus the functions $f_j(z)$ obtained from \eqref{gen_int_rep} are in the class $\mathcal{S}^*_{\rho}$. In particular, if we choose 
\[
q(z)=1+\sinh^{-1}(z) = 1+z-\dfrac{z^3}{6}+\dfrac{3z^5}{40}-\dfrac{5z^7}{112}...,
\]
then \eqref{gen_int_rep} gives
\begin{equation}\label{func_int_rep}
f_0(z) = z \exp\left(\int^z_0\frac{\sinh^{-1}(t)}{t}dt\right) = z + z^2 + \frac{z^3}{2} + \frac{z^4}{9} - \frac{z^5}{72} - \frac{z^6}{225}\cdots,
\end{equation}
which often acts as the extremal function for the class $\mathcal{S}^*_{\rho}$ yielding sharp results.

\begin{remark}
Note that $\sinh^{-1}(z) = \ln(z+\sqrt{1+z^2})$. Let $w=zf'(z)/f(z)$, where $f \in \mathcal{S}^*_{\rho}$. Then the class $\mathcal{S}^*_{\rho}$ can be alternatively represented by
$\exp(w-1) \prec z+\sqrt{1+z^2}$, where $z+\sqrt{1+z^2}$ represents the Crescent shaped domain \cite{raina}. Thus, there exists an exponential relation among the functions in the classes $\mathcal{S}^*_{\rho}$ and $\Delta^*$. 
\end{remark} 

In the present investigation, the geometrical properties of the function $1+\sinh^{-1}(z)$ are studied and certain inclusion properties as well as radius problems are established for the class $\mathcal{S}^*_{\rho}$.

\section{Properties of the function $1+\sinh^{-1}(z)$}
The current section deals with the study of some geometric properties of the function $1+\sinh^{-1}(z)$.

\begin{theorem}\label{convexfunction}
The function $\rho(z)=1+\sinh^{-1}(z)$ is a convex univalent function.  
\end{theorem}
\begin{proof}
Let $h(z) = \sinh^{-1}(z)$. Clearly, $h(0)=0$. Since $h'(z)=1/\sqrt{1+z^2}$ and $\sqrt{1+z^2} \prec \sqrt{1+z} \in \mathcal{P}$, where $\mathcal{P}$ is the Carath\'{e}odory class. Therefore, $1/\sqrt{1+z^2} \in \mathcal{P}$ which implies that $\RE h'(z) > 0$. Hence $\rho$ is univalent. 
Now a calculation yields
\begin{equation*}
1+\frac{zh''(z)}{h'(z)}=\frac{1}{1+z^2}.
\end{equation*}
Since  
\begin{equation*}
 \frac{1}{1+z^2}\prec \frac{1}{1+z} \in \mathcal{P}. 
\end{equation*}
Therefore, $\RE(1+zh''(z)/h'(z))>0$ which implies that $h$ (and thus $\rho$) is a convex univalent function. 
\qed \qed \end{proof}

\begin{remark}\label{real_symmetry}
Note that $\rho'(0)>0$ and the function $\varphi(z)= z+\sqrt{1+z^2}$ satisfies $\varphi(\bar{z})=\overline{\varphi(z)}$. Therefore, $\rho(\bar{z})=\overline{\rho(z)}$ and hence, the domain $\Omega_{\rho}=\rho(\mathbb{D})$ is symmetric about the real axis.
\end{remark}

\begin{theorem}\label{imag_symmetry}
The domain $\Omega_{\rho}$ is symmetric about the line $\RE(w)=1$.
\end{theorem}

\begin{proof}
Since $\Omega_{\rho}$ is symmetric about the real axis, the condition $0 \leq \theta \leq \pi/2$ is sufficient to prove our result. As we know that symmetry along imaginary axis for $f \in \mathcal{A}$ holds if $\RE(f(\theta)) = -\RE(f(\pi-\theta))$ and $\IM(f(\theta)) = \IM(f(\pi-\theta))$.
Now let $h(z) = \sinh^{-1}(z) = \ln(z+\sqrt{1+z^2})$. Then $\IM(h(z)) = \arg(z+\sqrt{1+z^2})$. For $z = re^{it}, t \in [0,\pi]$ and fixed $ r \in (0,1)$, we have the following expressions for $t \rightarrow \theta$ 
\[
\begin{array}{ll}
I_1 &= \arg\left(r(\cos\theta + i\sin\theta) + \sqrt{1+r^2(\cos(2\theta) + i\sin(2\theta))}\right)\\ 
&= \arg\left(z+\sqrt{1+z^2}\right),
\end{array}
\]
and for $t \rightarrow \pi-\theta$
\[
\begin{array}{ll}
I_2 &= \arg\left(r(\cos(\pi-\theta) + i\sin(\pi-\theta)) + \sqrt{1+r^2(\cos(2(\pi-\theta)) + i\sin(2(\pi-\theta)))}\right)\\
&= \arg\left(r(-\cos\theta + i\sin\theta) + \sqrt{1+r^2(\cos(2\theta) - i\sin(2\theta))}\right)\\
&= \arg\left(-\overline{z}+\sqrt{1+\overline{z}^2}\right).
\end{array}
\]
Now let us consider $(z+\sqrt{1+z^2})/({-\overline{z}+\sqrt{1+\overline{z}^2}})$. On rationalising the denominator, we get
\[
\displaystyle{\frac{z+\sqrt{1+z^2}}{-\overline{z}+\sqrt{1+\overline{z}^2}}} = \displaystyle{\frac{(z+\sqrt{1+z^2})(-z+\sqrt{1+z^2})}{(-\overline{z}+\sqrt{1+\overline{z}^2})(-z+\sqrt{1+z^2})}} = \displaystyle{\frac{1}{|-z+\sqrt{1+z^2}|^2}} = k > 0,
\]
where $k$ is some real positive constant. Thus,
\[
\begin{array}{ll}
&\arg\left( \displaystyle{\frac{z+\sqrt{1+z^2}}{-\overline{z}+\sqrt{1+\overline{z}^2}}} \right) = \arg(k) = 0\\
\Rightarrow & \arg\left(z+\sqrt{1+z^2}\right) = \arg\left(-\overline{z}+\sqrt{1+\overline{z}^2}\right)\\
\Rightarrow & I_1 = I_2. 
\end{array}
\]
Similarly, $\RE(h(\theta)) = -\RE(h(\pi-\theta))$ for $0 \leq \theta \leq \pi/2$. Hence, $h(z)$ is symmetric about the imaginary axis and thus, by translation property, $\rho(z)$ is symmetric about the line $\RE(w) = 1$.
\qed \end{proof}

Now using Theorem~\ref{imag_symmetry}, we obtain the next result:
\begin{corollary}\label{g-disk}
The disk $\{w: |w-1|\leq \sinh^{-1}(r)\}$ is contained in $\rho(|z|\leq r)$ and is maximal.
\end{corollary}
\begin{proof}
Since $\min_{|z|=r} |\sinh^{-1}(z)| = |\sinh^{-1}(-r)|=\sinh^{-1}(r)$ and hence the conclusion can be drawn at once.
\qed \end{proof}

\begin{figure}[t]
\centering
\includegraphics[scale=0.55]{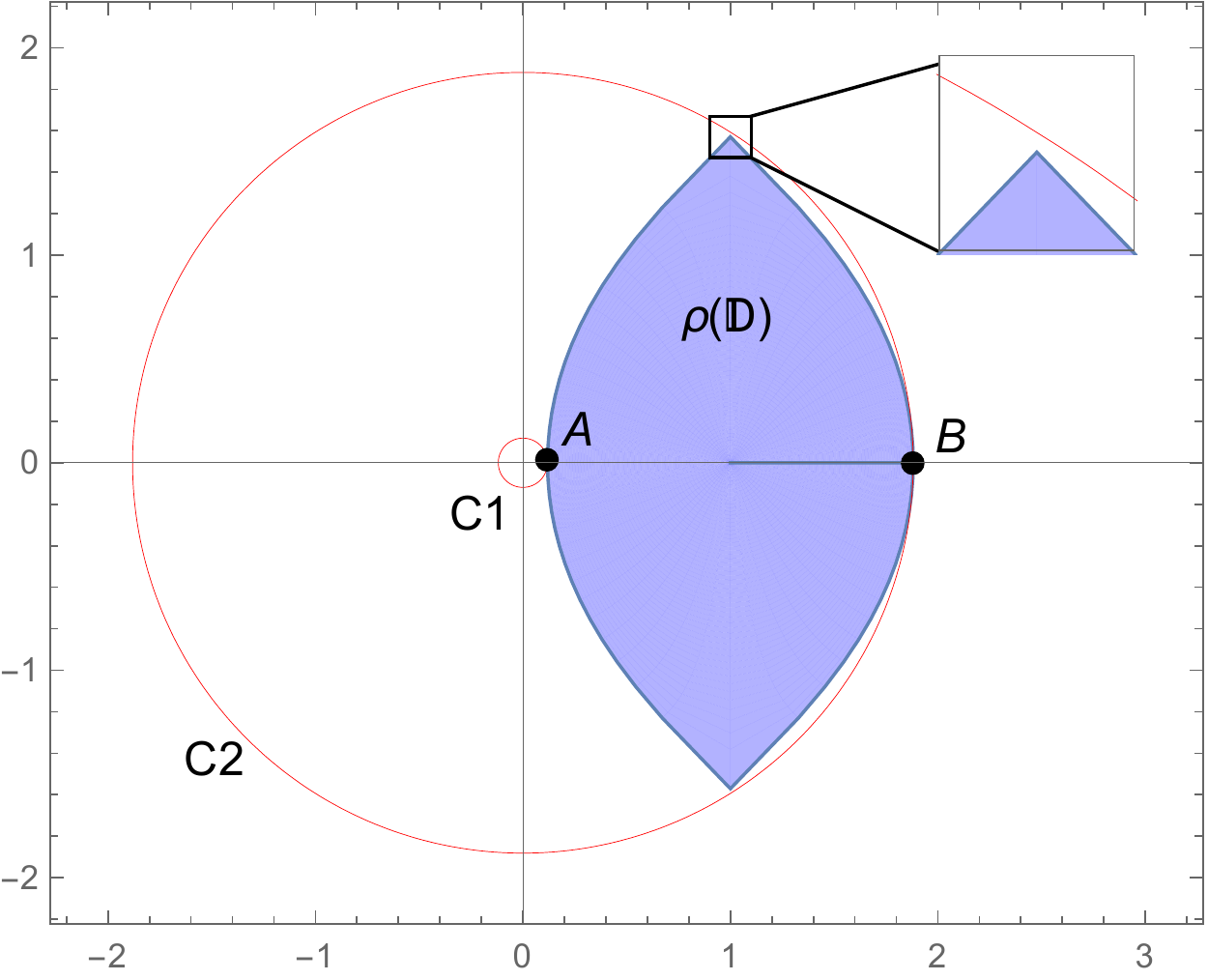}
\caption{ $\rho(\mathbb{D})$ lies in the annular region bounded between the circles $C1$ and $C2$.}
\label{fig:modulus}
\end{figure}

\begin{theorem}\label{geo_bounds}
We find that the following properties hold for $\rho(z)=1+\sinh^{-1}(z)$:
\begin{enumerate}[(i)]
    \item $\rho(-r) \leq \RE \rho(z) \leq \rho(r) \quad (|z|\leq r<1)$;
    \item $|\IM \rho(z)| \leq \pi/2 \quad (|z|\leq 1)$;
    \item $\rho(-r) \leq |\rho(z)| \leq \rho(r)\quad (|z|\leq r<1)$;
    \item $|\arg \rho(z)| \leq \tan^{-1}(1/t)$ where $t= \dfrac{4}{\pi} \sqrt{\sinh^{-1}(1)(1-\sinh^{-1}(1))}$.
\end{enumerate}
\end{theorem}

\begin{proof} (i) Since $\rho(z)$ is convex and typically real, the value of $\RE\rho(z)$ falls between $\lim_{\theta\rightarrow 0}\rho(re^\theta)$  and  $\lim_{\theta\rightarrow \pi}\rho(re^\theta)$, thus the result follows.
(ii) Using Theorem \ref{imag_symmetry}, it suffices to take $\theta \in [0,\pi/2]$. Then the inequality follows by letting $r$ tending to $1^-$ and observing that the function
\[
\IM \rho(z) = \arg \left(r \cos (\theta) + \sqrt{1+r^2 (\cos (2\theta)+i \sin (2\theta))}+i r\sin (\theta)\right)
\]
is strictly increasing in the interval $[0,\pi/2]$ and hence the result follows at once. (iii) The radially farthest and nearest points in $\rho(\mathbb{D})$ from origin are respectively $B$ and $A$ (see Figure \ref{fig:modulus}) and therefore the result obviously holds. Moreover we observe that these points $A$ and $B$ lie on the real line  and hence the bounds of $|\rho(z)|$ and $\RE \rho(z)$ coincide.
The proof of (iv) is evident from Theorem \ref{incl_rel}(iii) so skipped here.
\qed \end{proof}

Next we have the following important result:
\begin{lemma}
\label{disk_lem}
For $1-\sinh^{-1}(1)<a<1+\sinh^{-1}(1)$, let $r_a$ be given by
\[  
r_a=
 \left\{
   \begin{array}{lr}
    a-(1-\sinh^{-1}(1)), &  1-\sinh^{-1}(1)<a\leq1; \\
    1+\sinh^{-1}(1)-a,   & 1\leq a<1+\sinh^{-1}(1).
   \end{array}
 \right.
\]
Then
$
\{w : |w-a|<r_a\} \subset \Omega_{\rho}.
$
\end{lemma}
We omit the proof of Lemma \ref{disk_lem} as it directly follows from Theorem~\ref{imag_symmetry} and Corollary~\ref{g-disk}. 



 \begin{remark}
Evidently the domain $ \Omega_{\rho}$ is contained inside the disk
$\{w: |w-1|<\pi/2\}.$
 \end{remark}

\section{Inclusion Relations}
This section establishes some inclusion results involving the class $\mathcal{S}^*_{\rho}$ with some well-known classes.

We consider the class $M(\beta)$, first studied by Uralegaddi $et\; al.$ \cite{uralegaddi}, given by
\[
M(\beta) := \left\{ f \in \mathcal{A}: \RE\left( \frac{zf'(z)}{f(z)}\right) < \beta,\; z \in \mathbb{D}, \beta>1\right\},
\]
and another interesting class introduced by Kanas and  Wis\'niowska \cite{kanas} of $k$-starlike functions, denoted by $k-\mathcal{ST}$ and defined by 
\[
k-\mathcal{ST} \coloneqq \left\{ f \in \mathcal{A}: \RE\left( \frac{zf'(z)}{f(z)}\right) > k\left| \frac{zf'(z)}{f(z)}-1 \right|,\; z \in\mathbb{D}, k \geq 0 \right\}.
\]
Note that $\mathcal{S}^* = 0-\mathcal{ST}$ and $\mathcal{S}^*_p = 1-\mathcal{ST}$, where $\mathcal{S}^*_p$ is the class of parabolic starlike functions \cite{ronn}.

We establish the following inclusion relations for the class $\mathcal{S}^*_{\rho}$.

\begin{theorem}
\label{incl_rel}
The class $\mathcal{S}^*_{\rho}$ satisfies the following relationships:
\begin{enumerate}[(i)]
\item $\mathcal{S}^*_{\rho} \subset \mathcal{S}^*_{\alpha} \subset \mathcal{S}^*$ for $0 \leq \alpha \leq 1-\sinh^{-1}(1)$;
\item $\mathcal{S}^*_{\rho} \subset M(\beta)$ for $\beta \geq 1+\sinh^{-1}(1)$;
\item $\mathcal{S}^*_{\rho} \subset \mathcal{SS}^*(\gamma)$ for $(2/\pi)\tan^{-1}(1/t) \leq \gamma \leq 1$ where $t=\dfrac{4}{\pi} \sqrt{\sinh^{-1}(1)(1-\sinh^{-1}(1))}$;
\item $k-\mathcal{ST} \subset \mathcal{S}^*_{\rho}$ for $k \geq 1+1/\sinh^{-1}(1)$.
\end{enumerate}
\end{theorem}

\begin{proof}
Consider $f \in \mathcal{S}^*_{\rho}$ which implies $zf'(z)/f(z) \prec 1+\sinh^{-1}(z)$. By Theorem \ref{geo_bounds}, it is evident that for $z \in \mathbb{D}$,
\[
1-\sinh^{-1}(1) = \min_{|z|=1} \RE (1+\sinh^{-1}(z)) \leq \RE \frac{zf'(z)}{f(z)}
\]
and
\[
\RE \frac{zf'(z)}{f(z)} \leq \max_{|z|=1}\RE (1+\sinh^{-1}(z)) = 1+\sinh^{-1}(1).
\]
This proves (i) and (ii). For (iii), let $w \in \mathbb{C}$, $X = \RE(w)$, $Y = \IM(w)$, and $b = 1-\sinh^{-1}(1)$. Now consider the parabolic domain $\Gamma_P$ with the boundary curve $\partial\Gamma_P = \gamma_p: Y^2 = 4a(X-b)$. Then the focus $a$ of the smallest parabola $\gamma_p$ which contains $\Omega_{\rho}$ will touch the peak points $1 \pm i\pi/2$ of  $\mathcal{S}^*_{\rho}$ is $\pi^2/(16\sinh^{-1}(1))$. Let $P$ be any point on the parabola $\gamma_P$ with parametric coordinates $(b+at^2, 2at)$ such that the tangent OE at $P$ passes through origin for some parameter $t$. Let the equation of the tangent OE be $y=mx$, where $m = dy/dx = (dy/dt)/(dx/dt) = 1/t$. Therefore at P, we have 
\begin{equation*}
m =\frac{y}{x} \Rightarrow \frac{1}{t} = \frac{2at}{b+at^2},
\end{equation*}
 which yields 
 \begin{equation}\label{t} 
 t = \sqrt{\dfrac{b}{a}} = \dfrac{4}{\pi} \sqrt{\sinh^{-1}(1)(1-\sinh^{-1}(1))} 
 \end{equation} 
and the argument of the tangent at $P$ of $ \gamma_p$ is $\tan^{-1}(1/t)$. Since $\Omega_{\rho} \subset \Gamma_p$, it gives
\[
\left|\arg \frac{zf'(z)}{f(z)}\right| \leq \max_{|z|=1} \arg (\rho(z)) = \max \arg(\gamma_p) = \tan^{-1}(1/t),
\]
which demonstrates $f \in \mathcal{SS}^*((2/\pi)\tan^{-1}(1/t))$, where $t$ is given by \eqref{t}.

To show (iv), consider $f \in k-\mathcal{ST}$ along with the conic domain $\Gamma_k = \{w \in \mathbb{C}: \RE w > k|w-1|\}$. For $k>1$, let $\partial\Gamma_k$ represent the horizontal ellipse $\gamma_k: x^2 = k^2(x-1)^2+k^2y^2$ which may be rewritten as
\[
\frac{(x-x_0)^2}{a^2} + \frac{(y-y_0)^2}{b^2} = 1,
\]
where $x_0=k^2/(k^2-1),\, y_0=0,\, a=k/(k^2-1)$ and $b=1/\sqrt{k^2-1}$. For $\gamma_k \subset \Omega_{\rho}$, the condition $x_0+a\leq 1+\sinh^{-1}(1)$ must hold, or equivalently $k \geq 1+1/\sinh^{-1}(1)$. Since $\Gamma_{k_1} \subseteq \Gamma_{k_2}$ for $k_1 \geq k_2$, it follows that for $k \geq 1+1/\sinh^{-1}(1)$, $k-\mathcal{ST} \subset \mathcal{S}^*_{\rho}$. Figure \ref{fig:incl_rel} clearly depicts these relations.
\qed \end{proof}

\begin{figure}[t]
\begin{minipage}{0.55\textwidth}
\includegraphics[width=0.95\linewidth, height=8cm]{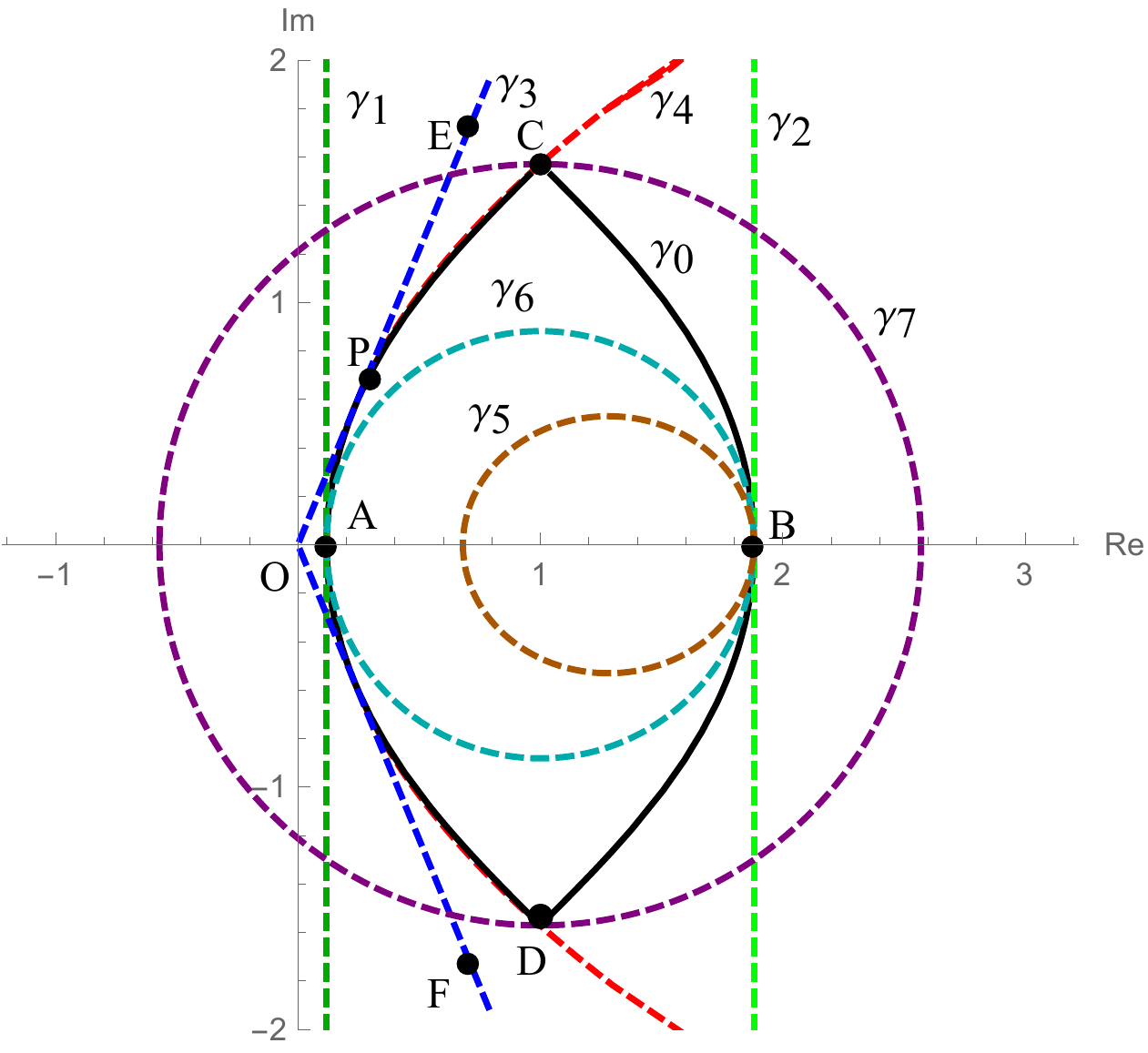}
\end{minipage}%
\begin{minipage}{0.45\textwidth}
\parbox{0.9\linewidth}{
  \small
  $\gamma_0: w = 1+\sinh^{-1}(z)$\\
  $\gamma_1: \RE(w)=1-\sinh^{-1}(1)$\\
  $\gamma_2: \RE(w)=1+\sinh^{-1}(1)$\\
  $\gamma_3: |\arg(w)|\leq \tan^{-1}(1/t)$\\
  where $t$ is given by \eqref{t}\\
  $\gamma_ 4= \gamma_p: Y^2 = 4a(X-1+\sinh^{-1}(1))$\\
  where $a = \pi^2/(16\sinh^{-1}(1))$\\
  $\gamma_5= \gamma_k: \RE(w)>k|w-1|$\\
  where $\; k=1+1/\sinh^{-1}(1)$\\
  $\gamma_6= D_L: |w-1|<\sinh^{-1}(1)$\\
  $\gamma_7= D_S: |w-1|<\pi/2$\\
  $\RE$(A) $= 1-\sinh^{-1}(1)$\\
  $\RE$(B) $= 1+\sinh^{-1}(1)$\\
  $\IM$(C) $= -\IM(D) = \pi/2$\\
  $\arg($E$)=-\arg($F$)= \tan^{-1}(1/t)$\\
  where $t$ is given by \eqref{t}\\
  O: origin\\
  P: point of tangency of $\gamma_4$ w.r.t. OE
  }
\end{minipage}

\caption{Boundary curves, depicting some inclusion relations for $w=1+\sinh^{-1}(z)$.}
\label{fig:incl_rel}
\end{figure}

For our next result, we consider $\mathcal{P}_n[C,D],$ the class of functions $p(z)$ of the form $1 + \sum_{k=n}^{\infty}c_{k}z^{k}$, satisfying $ p(z) \prec (1+Cz)/(1+Dz),$ where $-1 \leq D < C \leq 1$.
Denote by $\mathcal{P}_n(\alpha) := \mathcal{P}_n[1-2\alpha,-1] \; \text{and} \; \mathcal{P}_n:= \mathcal{P}_n(0)$. For n=1, $\mathcal{P} = \mathcal{P}_1$ is the Carath\'{e}odory class. We need the following lemmas:

\begin{lemma}
\label{p-nAlpha_lem}\normalfont{\cite{shah}}
For $p \in \mathcal{P}_n(\alpha)$, we have
\[
\left|\frac{zp'(z)}{p(z)}\right| \leq \frac{2(1-\alpha)nr^n}{(1-r^n)(1+(1-2\alpha)r^n)}, \; (|z|=r).
\]
\end{lemma}

\begin{lemma}
\label{p-nCD_lem}\normalfont{\cite{ravi-ron}}
For $p\in\mathcal{P}_n[C,D]$, we have
\[
\left|p(z)-\frac{1-CDr^{2n}}{1-D^2r^{2n}}\right| \leq \frac{(C-D)r^n}{1-D^2r^{2n}}, \; (|z|=r).
\]
Especially, for $p \in \mathcal{P}_n(\alpha)$, we have
\[
\left| p(z) - \frac{1+(1-2\alpha)r^{2n}}{1-r^{2n}}\right| \leq \frac{2(1-\alpha)r^n}{1-r^{2n}}, \; (|z|=r).
\]
\end{lemma}
\begin{theorem}
Let $-1 < D < C \leq 1$. If either of the following two conditions holds:
\begin{enumerate}[(i)]
\item $(1-\sinh^{-1}(1))(1-D^2) < 1-CD\leq 1-D^2$ and $C-D\leq (1-D)\sinh^{-1}(1)$;
\item $1-D^2 \leq 1-CD < (1+\sinh^{-1}(1))(1-D^2)$ and $C-D\leq (1+D)\sinh^{-1}(1)$.
\end{enumerate}
Then $\mathcal{S}^{*}[C,D]\subset \mathcal{S}^*_{\rho}$.
\end{theorem}

\begin{proof}
Let $f\in\mathcal{S}^{*}[C,D]$ which implies $zf'(z)/f(z)\in\mathcal{P}[C,D]$. Using Lemma \ref{p-nCD_lem} we have
\begin{equation}
\label{p-thm-eq}
\left|\frac{zf'(z)}{f(z)}-\frac{1-CD}{1-D^2}\right| \leq \frac{(C-D)}{1-D^2}.
\end{equation}
Let $a=(1-CD)/(1-D^2)$ and assume that (i) holds. Now multiplying $1+D$ and dividing by $(1-D^2)$ on either sides of the inequality $(C-D) \leq (1-D)\sinh^{-1}(1)$ gives $(C-D)/(1-D^2)\leq a-(1-\sinh^{-1}(1))$ on simplification. Also, the inequality $(1-\sinh^{-1}(1))(1-D^2) < 1-CD \leq 1-D^2$ is equivalent to $1-\sinh^{-1}(1)<(1-CD)/(1-D^2)\leq 1$. Therefore, from \eqref{p-thm-eq} we find $w=zf'(z)/f(z)$ is contained inside the disk $|w-a|<r_a$, where $r_a=a-(1-\sinh^{-1}(1))$ and $1-\sinh^{-1}(1) < a\leq 1$. Hence $f\in\mathcal{S}^*_{\rho}$ by Lemma \ref{disk_lem}. A similar proof can be shown when (ii) holds. \qed \end{proof}

\section{Radius Problems}
\noindent In this section, radius results for various subclasses of $\mathcal{A}$ are established. We begin by determining sharp $\mathcal{S}^*_{\alpha}\;(0\leq\alpha<1),\; \mathcal{M}(\beta)\;(\beta>1) $ and $k-\mathcal{ST}$-radii $(k \geq 0)$ for the class $\mathcal{S}^*_{\rho}$. Using Theorem \ref{incl_rel}, we can establish that $R_{\mathcal{S}^*_{\alpha}}(\mathcal{S}^*_{\rho})= R_{M(\beta)}(\mathcal{S}^*_{\rho})= 1$ for $0\leq\alpha\leq 1-\sinh^{-1}(1)$ and $\beta > 1+\sinh^{-1}(1)$.

\begin{theorem}
If $f \in \mathcal{S}^*_{\rho}$, then the following results hold:
\begin{enumerate}[(i)]
\item For $1-\sinh^{-1}(1)\leq\alpha<1$, we have $f \in \mathcal{S}^*_{\alpha}$ in $|z| \leq \sinh(1-\alpha)$.
\item For $1<\beta\leq 1+\sinh^{-1}(1)$, we have $f \in \mathcal{M}(\beta)$ in $|z| \leq \sinh(\beta-1)$.
\item For $k>0$, we have $f \in k-\mathcal{ST}$ in $|z| \leq \sinh(1/(k+1))$.
\end{enumerate}
The results are sharp.
\end{theorem}

\begin{proof}
Since $f \in \mathcal{S}^*_{\rho},\; zf'(z)/f(z) \prec 1+\sinh^{-1}(z)$ and hence for $|z|=r<1$ Theorem \ref{geo_bounds} gives
\[
1-\sinh^{-1}(r) \leq \RE \frac{zf'(z)}{f(z)} \leq 1+\sinh^{-1}(r),
\]
thereby validating the first two parts. Also, the constants $\sinh(1-\alpha)$ and $\sinh(\beta-1)$ are optimal for the function $f_0$ given by \eqref{func_int_rep}. Now to prove (iii), note that $f \in k-\mathcal{ST}$ in $|z|<r$, if 
\[
\RE (1+\sinh^{-1}(w(z))) \geq k|1+\sinh^{-1}(w(z))-1| = k|\sinh^{-1}(w(z))|.
\]
Here $w$ denotes the Schwarz function. Since $\RE (1+\sinh^{-1}(w(z))) \geq 1-\sinh^{-1}(r)$ and $|\sinh^{-1}(w(z))| \leq \sinh^{-1}(r)$, 
 the inequality $\RE (1+\sinh^{-1}(w(z)) \geq k|\sinh^{-1}(w(z))|$ holds whenever $ 1-\sinh^{-1}(r) \geq k\sinh^{-1}(r)$,  which implies $r\leq\sinh(1/(1+k))$. For the function $f_0$ given by \eqref{func_int_rep} and for $z_0=-\sinh(1/(1+k))$, we have
\[
\RE \frac{z_0f'_0(z_0)}{f_0(z_0)} = \RE (1+\sinh^{-1}(z_0)) = \frac{k}{k+1} = k|\sinh^{-1}(z_0)| = k\left|\frac{z_0f'_0(z_0)}{f_0(z_0)}-1\right|.
\]
This concludes the proof.
\qed \end{proof}

\begin{corollary}
Substituting $k=1$ in part (iii) above, we find that $f \in \mathcal{S}^*_{\rho}$ is parabolic starlike \normalfont{\cite{ronn}} in $|z| \leq \sinh(1/2)$. 
\end{corollary}

\noindent In the next result, we find the $\mathcal{K}_{\alpha}$-radius for the class $\mathcal{S}^*_{\rho}$.
\begin{theorem}
Let $f\in \mathcal{S}^*_{\rho}$. Then $f\in \mathcal{K}_{\alpha}$ in $|z|<r_{\alpha}$, where $r_{\alpha}$ is the least positive root of
\begin{equation}
\label{rconv}
(1-r^2)\sqrt{1+r^2}\left(1-\sinh^{-1}(r)\right)\left(1-\alpha-\sinh^{-1}(r)\right)-r = 0 \quad (0\leq\alpha<1).
\end{equation}
\end{theorem}

\begin{proof}
Let $f \in \mathcal{S}^*_{\rho}$ and $w$ be a Schwarz function. Then $zf'(z)/f(z)= 1+\sinh^{-1}(w(z))$ such that
\[
1+\frac{zf''(z)}{f'(z)} = 1+\sinh^{-1}(w(z)) + \frac{zw'(z)}{(1+\sinh^{-1}(w(z)))\sqrt{1+w^2(z)}}
\]
which yields
\[
\RE\left(1+\frac{zf''(z)}{f'(z)}\right) \geq \RE\left(1+\sinh^{-1}(w(z))\right) - \left|\frac{zw'(z)}{(1+\sinh^{-1}(w(z)))\sqrt{1+w^2(z)}}\right|.
\]
We know for the Schwarz function $w$, the inequality $|w'(z)| \leq (1-|w(z)|^2)/(1-|z|^2)$ holds. Thus 
we observe that
\begin{align*}
\RE\left(1+\frac{zf''(z)}{f'(z)}\right) &\geq 1-\sinh^{-1}(|z|) - \frac{|z|(1-|w(z)|^2)}{(1-\sinh^{-1}(|z|))(1-|z|^2)\sqrt{1+|z|^2}}\\
&\geq 1-\sinh^{-1}(|z|) - \frac{|z|}{(1-\sinh^{-1}(|z|))(1-|z|^2)\sqrt{1+|z|^2}}.
\end{align*}
Now consider the function $q(r):=1-\sinh^{-1}(r)-r/\left((1-\sinh^{-1}(r))(1-r^2)\sqrt{1+r^2}\right)$. This is a decreasing function in $[0,1)$ with $q(0)=1$. Therefore $\RE(1+zf''(z)/f'(z))>\alpha$ in $|z|<r_{\alpha}<1$, where $r_{\alpha}$ is given as the least positive root of the equation $q(r)=\alpha$, which is same as \eqref{rconv} and hence the result. 
\qed \end{proof}

\begin{remark} Note for $\alpha = 0$, $r_0 \approx 0.37198$ which is not sharp, so the result can be further improved. The sharp $\mathcal{K}_{0}$-radius for the class $\mathcal{S}^*_{\rho}$ is $r_0 \approx 0.400435$, which we can guess graphically but a mathematical proof is yet to derive.
\end{remark}

For our next theorems \ref{Kvn-rad-thm-1} - \ref{Kvn-rad-thm-4}, the following subclasses are required:\\
Let $\mathcal{S}^*_n[C,D]:= \{f\in\mathcal{A}_n  : zf'(z)/f(z)\in  \mathcal{P}_n[C,D] \}$. Also, let $\mathcal{S}^*_n(\alpha) := \mathcal{S}^*_n[1-2\alpha,-1] = \mathcal{A}_n \cap \mathcal{S}^*_{\alpha}\; \text{and}\; \mathcal{S}^*_{\rho,n} := \mathcal{A}_n \cap \mathcal{S}^*_{\rho}$.
Further, Ali $et\; al.$ \cite{ali12} studied the three classes $\mathcal{S}_n := \{f \in \mathcal{A}_n : f(z)/z \in \mathcal{P}_n\}, \, \mathcal{S}^*_n[C,D]$ and
\[
\mathcal{CS}_n(\alpha) := \left\{f \in \mathcal{A}_n : \frac{f(z)}{g(z)} \in \mathcal{P}_n, \; g \in \mathcal{S}^*_n(\alpha)\right\}.
\]
 Now we obtain the $\mathcal{S}^*_{\rho,n}$-radii for the classes defined above.

\begin{theorem}
\label{Kvn-rad-thm-1}
For the class $\mathcal{S}_n$, the sharp $\mathcal{S}^*_{\rho,n}$-radius  is given by:
\[
R_{\mathcal{S}^*_{\rho,n}}(\mathcal{S}_n) = \left(\frac{\sinh^{-1}(1)}{n + \sqrt{n^2 + \left(\sinh^{-1}(1)\right)^2}} \right)^{1/n}.
\]
\end{theorem}

\begin{proof}
Let $f \in \mathcal{S}_n$. Define $s: \mathbb{D} \rightarrow \mathbb{C}$ by $s(z) = f(z)/z$. Then $s \in \mathcal{P}_n$ and we can obtain $zf'(z)/f(z) - 1 = zs'(z)/s(z)$ from the above definition of $s$. Using Lemma \ref{disk_lem} and Lemma \ref{p-nAlpha_lem}, the following holds
\[
\left| \frac{zf'(z)}{f(z)} -1 \right| = \frac{zs'(z)}{s(z)} \leq \frac{2nr^n}{1-r^{2n}} \leq \sinh^{-1}(1),
\]
or equivalently $(\sinh^{-1}(1))r^{2n} + 2nr^n - \sinh^{-1}(1) \leq 0$. Therefore, the $\mathcal{S}^*_{\rho,n}$-radius of $\mathcal{S}_n$ is the least positive root of $(\sinh^{-1}(1))r^{2n} + 2nr^n - \sinh^{-1}(1)=0$ for $r\in(0,1)$. We can verify $\RE(f_0(z)/z)>0$ holds in $\mathbb{D}$ where $f_0(z) = z(1+z^n)/(1-z^n)$. Thus $f_0 \in \mathcal{S}_n$ and $zf'_0(z)/f_0(z) = 1 + 2nz^n/(1-z^{2n})$. Moreover, the result is sharp since at $z = R_{\mathcal{S}^*_{\rho,n}}(\mathcal{S}_n)$, we obtain
\[
\frac{zf'_0(z)}{f_0(z)} -1 = \frac{2nz^n}{1-z^{2n}} = \sinh^{-1}(1).
\]
The proof is complete.
\qed \end{proof}

Let $\mathcal{F}$ define the class of functions $f \in \mathcal{A}$ satisfying $f(z)/z \in \mathcal{P}$. The radius of univalence and starlikeness of the class $\mathcal{F}$ is $\sqrt{2}-1$, as shown in \cite{macgreg}.

\begin{corollary}
\label{Kvn-rad-thm-2}
For the class $\mathcal{F}$, the $\mathcal{S}^*_{\rho}$-radius  is stated as
\[
R_{\mathcal{S}^*_{\rho}}(\mathcal{F}) = -e+\sqrt{1+e^2} \approx 0.178105.
\]
\end{corollary}

\begin{theorem}
\label{Kvn-rad-thm-3}
For the class $\mathcal{CS}_n(\alpha)$, the sharp $\mathcal{S}^*_{\rho,n}$-radius is given by
\[
R_{\mathcal{S}^*_{\rho,n}}(\mathcal{CS}_n(\alpha)) = \left(\frac{\sinh^{-1}(1)}{n-\alpha+1 +\sqrt{(n-\alpha+1)^2 + (\sinh^{-1}(1)+2(1-\alpha))\sinh^{-1}(1)}}\right)^{1/n}.
\]
\end{theorem}

\begin{proof}
Let $f \in \mathcal{CS}_n(\alpha)$ and $g \in \mathcal{S}^*_n(\alpha)$. Considering $s(z)= f(z)/g(z),$ clearly indicates $s \in \mathcal{P}_n$. Also, it gives
\[
\frac{zf'(z)}{f(z)} = \frac{zs'(z)}{s(z)} + \frac{zg'(z)}{g(z)}.
\]
The use of Lemmas (\ref{p-nAlpha_lem} -- \ref{p-nCD_lem}) gives us
\begin{equation}
\label{eq_CSn-1}
\left| \frac{zf'(z)}{f(z)} - \frac{1+(1-2\alpha)r^{2n}}{1-r^{2n}} \right| \leq \frac{2(n-\alpha+1)r^n}{1-r^{2n}}.
\end{equation}

Considering $(1+(1-2\alpha)r^{2n})/(1-r^{2n}) \geq 1$, the relation $f \in \mathcal{S}^*_{\rho,n}$ follows from \eqref{eq_CSn-1} and Lemma \ref{disk_lem} if the subsequent inequality is true:
\[
\frac{1 + 2(n-\alpha+1)r^n + (1-2\alpha)r^{2n}}{1-r^{2n}} \leq 1+\sinh^{-1}(1)
\]
or equivalently, $(2-2\alpha+\sinh^{-1}(1))r^{2n} + 2(n-\alpha+1)r^n - \sinh^{-1}(1) \leq 0$ holds. Thus, the least positive root of
\[
(2-2\alpha+\sinh^{-1}(1))r^{2n} + 2(n-\alpha+1)r^n - \sinh^{-1}(1) = 0
\]
gives the $\mathcal{S}^*_{\rho,n}$-radius for the class $\mathcal{CS}_n(\alpha)$. 
Next examine the following functions 
\begin{equation}
\label{eq_CSn-2}
f_0(z) = \frac{z(1+z^n)}{(1-z^n)^{(n+2-2\alpha)/n}} \; \text{and} \; g_0(z) = \frac{z}{(1-z^n)^{2(1-\alpha)/n}},
\end{equation}
which implies $f_0(z)/g_0(z) = (1+z^n)/(1-z^n)$ and $zg'_0(z)/g_0(z) = (1+(1-2\alpha)z^n)/(1-z^n)$. Moreover, it is obvious that $\RE(f_0(z)/g_0(z))>0$ and $\RE(zg'_0(z)/g_0(z))>\alpha$ in the unit disk $\mathbb{D}$. Hence $f_0 \in \mathcal{CS}_n(\alpha)$. At $z = R_{\mathcal{S}^*_{\rho,n}}(\mathcal{CS}_n(\alpha))$, the function $f_0$ defined in \eqref{eq_CSn-2} satisfies
\[
\frac{zf'_0(z)}{f_0(z)} = \frac{1 + 2(n-\alpha+1)z^n + (1-2\alpha)z^{2n}}{1-z^{2n}} = 1+\sinh^{-1}(1),
\]
which accomplish sharpness of the result.
\qed \end{proof}

\begin{theorem}
\label{Kvn-rad-thm-4}
For the class $\mathcal{S}^*_n[C,D]$, the $\mathcal{S}^*_{\rho,n}$-radius is given by
\[
R_{\mathcal{S}^*_{\rho,n}}(\mathcal{S}^*_n[C,D]) = \left\{
	\begin{array}{ll}
	  \min \{ 1;R_1\}, & -1 \leq D < 0 < C \leq 1;\\
	  \min \{1;R_2\}, & 0 < D < C \leq 1,	
	\end{array}	
	\right.
\]
where
\[
R_1 := \left( \frac{2\sinh^{-1}(1)}{C-D + \sqrt{(C-D)^2 + 4(D^2(1+\sinh^{-1}(1))-CD)\sinh^{-1}(1)}}\right)^{1/n}
\]
and
\[
R_2 := \left( \frac{2\sinh^{-1}(1)}{C-D + \sqrt{(C-D)^2 + 4(D^2(\sinh^{-1}(1) - 1) + CD)\sinh^{-1}(1)}}\right)^{1/n}.
\]
\end{theorem}

\begin{proof}
Let $f \in \mathcal{S}^*_n[C,D]$. From Lemma \ref{p-nCD_lem}, we have
\begin{equation}
\label{eq_S*nCD}
\left| \frac{zf'(z)}{f(z)} -b \right| \leq \frac{(C-D)r^n}{1-D^2r^{2n}},
\end{equation}
where $b = (1-CDr^{2n})/(1-D^2r^{2n}), \; |z|=r,$ represents the center of the disk. We infer $b \geq 1$ for $-1 \leq D < 0 < C \leq 1$. From Lemma \ref{disk_lem}, $f \in \mathcal{S}^*_{\rho,n}$ depends on whether following condition is true: 
\[
\frac{1 + (C-D)r^n - CDr^{2n}}{1-D^2 r^{2n}} \leq 1+\sinh^{-1}(1),
\]
which reduces to
\[
r \leq \left( \frac{2\sinh^{-1}(1)}{C-D + \sqrt{(C-D)^2 + 4(D^2(1+\sinh^{-1}(1)) - CD)\sinh^{-1}(1)}}\right)^{1/n} = R_1.
\]
Further, taking $D=0$, we get $b=1$. Then \eqref{eq_S*nCD} yields
\[
\left| \frac{zf'(z)}{f(z)} -1 \right| \leq Cr^n, \; (0 < C \leq 1).
\]
Now applying Lemma \ref{disk_lem} with $a=1$ gives $f \in \mathcal{S}^*_{\rho,n}$, if $r \leq ((\sinh^{-1}(1))/C)^{1/n}$.

For $0 < D < C \leq 1$, we have $b < 1$. Thus, using Lemma \ref{disk_lem} and \eqref{eq_S*nCD}, we have $f \in \mathcal{S}^*_{\rho,n}$ if the following holds:
\[
\frac{CDr^{2n} + (C-D)r^n -1}{1-D^2 r^{2n}} \leq \sinh^{-1}(1) - 1,
\]
or equivalently, if
\[
r \leq \left( \frac{2\sinh^{-1}(1)}{C-D + \sqrt{(C-D)^2 + 4(D^2(\sinh^{-1}(1) - 1) + CD)\sinh^{-1}(1)}}\right)^{1/n} = R_2.
\]
This concludes the proof.
\qed \end{proof}

The next theorem establishes radius results for some well-known classes mentioned earlier.

\begin{theorem}
The sharp $\mathcal{S}^*_{\rho}$-radii for the classes $\mathcal{S}^*_{L}, \mathcal{S}^*_{RL}, \mathcal{S}^*_{C}, \mathcal{S}^*_{e}, \Delta^* \, \text{and} \, \mathcal{BS}^*(\alpha)$ are:
\begin{enumerate}[(i)]
\item  $R_{\mathcal{S}^*_{\rho}}(\mathcal{S}^*_{L}) = \sinh^{-1}(1)(2-\sinh^{-1}(1)) \approx 0.985928$.
\item  $R_{\mathcal{S}^*_{\rho}}(\mathcal{S}^*_{RL}) = \dfrac{\left(2+(1+\sqrt{2})\sinh^{-1}(1)\right)\sinh^{-1}(1)}{5-3\sqrt{2} + \left(4(\sqrt{2}-1)+2\sinh^{-1}(1)\right)\sinh^{-1}(1)} \approx 0.964694$.
\item  $R_{\mathcal{S}^*_{\rho}}(\mathcal{S}^*_{C}) = \dfrac{1}{2}\left(\sqrt{2\left(2+3\sinh^{-1}(1)\right)}-2\right) \approx 0.523831$.
\item  $R_{\mathcal{S}^*_{\rho}}(\mathcal{S}^*_{e}) = \ln(1+\sinh^{-1}(1)) \approx 0.632002.$
\item  $R_{\mathcal{S}^*_{\rho}}(\Delta^*) = \dfrac{\sinh^{-1}(1)(2+\sinh^{-1}(1))}{2(1+\sinh^{-1}(1))} \approx 0.674924$.
\item $R_{\mathcal{S}^*_{\rho}}(\mathcal{BS}^*(\alpha)) = \dfrac{-1+\sqrt{1+\alpha\left(2\sinh^{-1}(1)\right)^2}}{2\alpha\sinh^{-1}(1)}, \; \alpha \in [0,1]$.

\end{enumerate}
\end{theorem}

\begin{proof}
\begin{enumerate}[(i)]
\item Suppose $f \in \mathcal{S}^*_{L}$. We have $zf'(z)/f(z) \prec \sqrt{1+z}$. When $|z|=r$, we obtain
\[
\left|\frac{zf'(z)}{f(z)} -1 \right| \leq 1 - \sqrt{1-r} \leq \sinh^{-1}(1),
\]
such that $r \leq (2-\sinh^{-1}(1))\sinh^{-1}(1) = R_{\mathcal{S}^*_{\rho}}(\mathcal{S}^*_{L})$ holds. Next examine the function
\[
f_0(z) = \frac{4z}{\left(1+\sqrt{1+z}\right)^2} {e^{2\left(\sqrt{1+z}-1\right)}}.
\]
Since $zf_0'(z)/f_0(z) = \sqrt{1+z}$, it follows that $f_0 \in \mathcal{S}^*_{L}$. As $zf_0'(z)/f_0(z) -1 = -\sinh^{-1}(1)$ is obtained at $z=-R_{\mathcal{S}^*_{\rho}}(\mathcal{S}^*_{L})$, the result is sharp.

\item Suppose $f \in \mathcal{S}^*_{RL}$, we obtain
\[
\frac{zf'(z)}{f(z)} \prec \sqrt{2} - (\sqrt{2}-1) \sqrt{\frac{1-z}{(1+2(\sqrt{2}-1)z)}}.
\]
For $|z|=r$, the subsequent inequality holds
\[
\left|\frac{zf'(z)}{f(z)}-1 \right| \leq 1 - \sqrt{2} + (\sqrt{2}-1) \sqrt{\frac{1+r}{(1-2(\sqrt{2}-1)r)}} \leq \sinh^{-1}(1),
\]
provided
\[
r \leq \frac{\left(2+(1+\sqrt{2})\sinh^{-1}(1)\right)\sinh^{-1}(1)}{5-3\sqrt{2} + \left(4(\sqrt{2}-1)+2\sinh^{-1}(1)\right)\sinh^{-1}(1)} = R_{\mathcal{S}^*_{\rho}}(\mathcal{S}^*_{RL}).
\]
Next observe the following function defined as
\[
f_0(z)= z \exp\left(\int^z_0\frac{q_0(t)-1}{t}dt \right),
\]
where
\[
q_0(t) = \sqrt{2} - (\sqrt{2}-1) \sqrt{\frac{1-t}{(1+2(\sqrt{2}-1)t)}}.
\]
From the definition of $f_0$, at $z=-R_{\mathcal{S}^*_{\rho}}(\mathcal{S}^*_{RL})$, we have
\[
\frac{zf_0'(z)}{f_0(z)} = \sqrt{2} - (\sqrt{2}-1) \sqrt{\frac{1-z}{(1+2(\sqrt{2}-1)z)}} = 1-\sinh^{-1}(1).
\]
which confirms the sharpness.

\item Suppose $f \in \mathcal{S}^*_{C}$. So $zf'(z)/f(z) \prec 1+4z/3+2z^2/3$. This gives
\[
\left|\frac{zf'(z)}{f(z)}-1\right| \leq \frac{4r}{3} + \frac{2r^2}{3} \leq \sinh^{-1}(1), \; |z|=r,
\]
for $r \leq \frac{1}{2}\left(\sqrt{2\left(2+3\sinh^{-1}(1)\right)}-2\right) = R_{\mathcal{S}^*_{\rho}}(\mathcal{S}^*_{C})$. The sharpness of the result is established using the subsequent function
\[
f_0(z) = z \exp\left(\frac{4z+z^2}{3}\right),
\]
where $zf_0'(z)/f_0(z) = 1+(4z+2z^2)/3$ yields $f_0 \in \mathcal{S}^*_{\rho}$, and substituting $z = R_{\mathcal{S}^*_{\rho}}(\mathcal{S}^*_{C})$ gives $zf_0'(z)/f_0(z) = 1+\sinh^{-1}(1)$, thereby proving the sharpness.

\item Suppose $f \in \mathcal{S}^*_{e}$, we have $zf'(z)/f(z) \prec e^z,$ which yields
\[
\left| \frac{zf'(z)}{f(z)}-1 \right| \leq e^r-1 \leq \sinh^{-1}(1) \;\text{holds in} \; |z|=r,
\]
provided $r \leq \ln(1+\sinh^{-1}(1)) = R_{\mathcal{S}^*_{\rho}}(\mathcal{S}^*_{e})$.
Now Consider
\[
f_0(z)= z \exp\left(\int^z_0\frac{e^t -1}{t}dt \right).
\]
Since $zf_0'(z)/f_0(z) = e^z$,   $f_0 \in \mathcal{S}^*_{e}$, and at $z=R_{\mathcal{S}^*_{\rho}}(\mathcal{S}^*_{e})$, we have $zf_0'(z)/f_0(z) = 1+\sinh^{-1}(1)$, which shows the sharpness of the result.

\item Suppose $f \in \Delta^*$ which gives $zf'(z)/f(z) \prec z+ \sqrt{1+z^2}$. Then,
\[
\left| \frac{zf'(z)}{f(z)}-1 \right|\leq r+\sqrt{1+r^2} -1 \leq \sinh^{-1}(1), \, |z|=r,
\]
for $r \leq \dfrac{\sinh^{-1}(1)(2+\sinh^{-1}(1))}{2(1+\sinh^{-1}(1))} = R_{\mathcal{S}^*_{\rho}}(\Delta^*)$.
For sharpness, define $f_0$ as
\[
f_0(z)= z \exp\left(\int^z_0\frac{t+\sqrt{1+t^2}-1}{t}dt \right).
\]
Since $zf_0'(z)/f_0(z) = z+\sqrt{1+z^2}$,  $f_0 \in \Delta^*$, so at $z=R_{\mathcal{S}^*_{\rho}}(\Delta^*)$, we have $zf_0'(z)/f_0(z)$ $= 1+\sinh^{-1}(1)$ which shows the sharpness of the result.

\item Suppose $f \in \mathcal{BS}^*(\alpha),\;\alpha \in [0,1]$, which gives $zf'(z)/f(z) \prec 1+ z/(1-\alpha z^2)$. Then,
\[
\left| \frac{zf'(z)}{f(z)}-1 \right|\leq \frac{r}{1-\alpha r^2} \leq \sinh^{-1}(1), \, |z|=r,
\]
for $r \leq \dfrac{-1+\sqrt{1+\alpha\left(2\sinh^{-1}(1)\right)^2}}{2\alpha\sinh^{-1}(1)} = R_{\mathcal{S}^*_{\rho}}(\mathcal{BS}^*(\alpha)), \, \alpha \in (0,1]$. For $\alpha=0,\, r \leq \sinh^{-1}(1)$. 
Next examine the function $f_0$ defined as
\[
f_0(z)= z \left(\frac{1+\sqrt{\alpha}z}{1-\sqrt{\alpha}z}\right)^{1/(2\sqrt{\alpha})}.
\]
Since $zf_0'(z)/f_0(z) = 1+ z/(1-\alpha z^2)$,  $f_0 \in (\mathcal{BS}^*(\alpha))$, so at $z=-R_{\mathcal{S}^*_{\rho}}((\mathcal{BS}^*(\alpha)))$, we have $zf_0'(z)/f_0(z)$ $= 1-\sinh^{-1}(1)$, which ensures sharpness of the result. 
\end{enumerate}

\noindent Note that $ R_{\mathcal{S}^*_{\rho}}(\mathcal{BS}^*(1)) = \left(-1+\sqrt{1+(2\sinh^{-1}(1))^2}/(2\sinh^{-1}(1))\right) \approx 0.58241$ and $R_{\mathcal{S}^*_{\rho}}(\mathcal{BS}^*(0))$ $= \sinh^{-1}(1) \approx 0.881374$.
\qed \end{proof}

Next we present some radius problems for certain classes of functions expressed as ratio of functions:
\[
\mathcal{F}_1 := \left\{ f \in \mathcal{A}_n : \RE \left(\frac{f(z)}{g(z)}\right) > 0 \;\text{and}\; \RE \left(\frac{g(z)}{z}\right) > 0,\; g \in \mathcal{A}_n\right\},
\]
\[
\mathcal{F}_2 := \left\{ f \in \mathcal{A}_n : \RE \left(\frac{f(z)}{g(z)}\right) > 0 \;\text{and}\; \RE \left(\frac{g(z)}{z}\right) > 1/2,\; g \in \mathcal{A}_n\right\},
\]
and
\[
\mathcal{F}_3 := \left\{ f \in \mathcal{A}_n : \left|\frac{f(z)}{g(z)} -1 \right| < 1 \;\text{and}\; \RE \left(\frac{g(z)}{z}\right) > 0,\; g \in \mathcal{A}_n\right\}.
\]

\begin{theorem} \label{ratio_func}
For functions in the classes $\mathcal{F}_1,\, \mathcal{F}_2$ and $ \mathcal{F}_3$, the sharp $\mathcal{S}^*_{\rho,n}$-radii respectively, are:
\begin{enumerate}[(i)]
\item $R_{\mathcal{S}^*_{\rho,n}}(\mathcal{F}_1) = \left(\dfrac{\sqrt{4n^2 + (\sinh^{-1}(1))^2}-2n}{\sinh^{-1}(1)}\right)^{1/n}$.
\item $R_{\mathcal{S}^*_{\rho,n}}(\mathcal{F}_2) = \left( \dfrac{\sqrt{9n^2 + 4\sinh^{-1}(1)(n + \sinh^{-1}(1))} -3n}{2(n+\sinh^{-1}(1))} \right)^{1/n}$.
\item $R_{\mathcal{S}^*_{\rho,n}}(\mathcal{F}_3) = R_{\mathcal{S}^*_{\rho,n}}(\mathcal{F}_2)$.
\end{enumerate}
\end{theorem}

\begin{proof}
\begin{enumerate}[(i)]
\item Let $f \in \mathcal{F}_1$ and consider the functions $s,d: \mathbb{D} \rightarrow \mathbb{C}$ where $s(z)=f(z)/g(z)$ and $d(z)=g(z)/z$. Clearly, $s,d \in \mathcal{P}_n$. As $f(z)=zd(z)s(z)$, applying Lemma \ref{p-nAlpha_lem} here gives
\[
\left|\frac{zf'(z)}{f(z)} -1 \right| \leq \frac{4nr^n}{1-r^{2n}} \leq \sinh^{-1}(1)
\]
such that 
\[
r \leq \left(\frac{\sqrt{4n^2 + (\sinh^{-1}(1))^2}-2n}{\sinh^{-1}(1)}\right)^{1/n} = R_{\mathcal{S}^*_{\rho,n}}(\mathcal{F}_1)
\]
holds. Next examine the functions
\[
f_0(z)= z \left(\frac{1+z^n}{1-z^n}\right)^2 \; \text{and} \; g_0(z) = z \left(\frac{1+z^n}{1-z^n}\right).
\]
Evidently, $\RE(f_0(z)/g_0(z))>0$ and $\RE(g_0(z)/z)>0$, which implies $f_0 \in \mathcal{F}_1$. Further calculation yields at $z = R_{\mathcal{S}^*_{\rho,n}}(\mathcal{F}_1)e^{i\pi/n}$
\[
\frac{zf'_0(z)}{f_0(z)} = 1 + \frac{4nz^n}{1-z^{2n}} = 1 - \sinh^{-1}(1),
\]
which validates the result is sharp.

\item Let $f \in \mathcal{F}_2$ and consider the functions $s,d: \mathbb{D} \rightarrow \mathbb{C}$ where $s(z)=f(z)/g(z)$ and $d(z)=g(z)/z$. Clearly, $s \in \mathcal{P}_n(1/2)$ and $d \in \mathcal{P}_n$. As $f(z)=zd(z)s(z)$, applying \ref{p-nAlpha_lem} here gives
\[
\left|\frac{zf'(z)}{f(z)} -1 \right| \leq \frac{2nr^n}{1-r^{2n}} + \frac{nr^n}{1-r^n} = \frac{3nr^n + nr^{2n}}{1-r^{2n}} \leq \sinh^{-1}(1),
\]
whenever
\[
r \leq \left( \frac{\sqrt{9n^2 + 4\sinh^{-1}(1)(n + \sinh^{-1}(1))} -3n}{2(n+\sinh^{-1}(1))} \right)^{1/n} = R_{\mathcal{S}^*_{\rho,n}}(\mathcal{F}_2).
\]
Therefore, $f \in \mathcal{S}^*_{\rho,n}$ holds for $r \leq R_{\mathcal{S}^*_{\rho,n}}(\mathcal{F}_2)$.
Next observe that $\RE(g_0(z)/z)>1/2$ while $\RE(f_0(z)/g_0(z))>0$ for the functions
\[
f_0(z) = \frac{z(1+z^n)}{(1-z^n)^2} \; \text{and} \; g_0(z) = \frac{z}{1-z^n}.
\]
Therefore $f_0 \in \mathcal{F}_2$ which verifies the sharpness for $z = R_{\mathcal{S}^*_{\rho,n}}(\mathcal{F}_2)$ such that
\[
\frac{zf'_0(z)}{f_0(z)} -1 = \frac{3nz^n + nz^{2n}}{1-z^{2n}} = \sinh^{-1}(1).
\]

\item Let $f \in \mathcal{F}_3$ and consider the functions $s,d: \mathbb{D} \rightarrow \mathbb{C}$ where $s(z)=g(z)/f(z)$ and $d(z) = g(z)/z$. Then $d \in \mathcal{P}_n$. We can verify that $|1/s(z)-1|<1$ holds whenever $\RE(s(z))>1/2$ and therefore $s \in \mathcal{P}_n(1/2)$. As $f(z)=zd(z)/s(z)$, on applying Lemma \ref{p-nAlpha_lem}, we obtain
\[
\left|\frac{zf'(z)}{f(z)} -1 \right| \leq \frac{3nr^n + nr^{2n}}{1-r^{2n}} \leq \sinh^{-1}(1).
\]
The rest of the proof is omitted as it is analogous to proof of Theorem \ref{ratio_func}(ii). The sharpness can be verified as follows. Examine the functions
\[
f_0(z) = \frac{z(1+z^n)^2}{1-z^n} \; \text{and} \; g_0(z) = \frac{z(1+z^n)}{1-z^n}.
\]
Using above definitions of $f_0$ and $g_0$, we see that
\[
\RE\left(\frac{g_0(z)}{f_0(z)}\right) = \RE\left(\frac{1}{1+z^n}\right) > \frac{1}{2} \; \text{and} \; \RE\left(\frac{g_0(z)}{z}\right) = \RE\left(\frac{1+z^n}{1-z^n}\right) > 0,
\]
and therefore, $f_0 \in \mathcal{F}_3$. Now at $z = R_{\mathcal{S}^*_{\rho,n}}(\mathcal{F}_3) e^{i\pi/n}$, we obtain
\[
\frac{zf'_0(z)}{f_0(z)} -1 = \frac{3nz^n - nz^{2n}}{1-z^{2n}} = -\sinh^{-1}(1),
\]
which serves as validation for the sharp result.
\end{enumerate}
This concludes the proof.
\qed \end{proof}

\end{document}